\documentclass{amsart}%
\usepackage{amsfonts}
\usepackage{amsmath}
\usepackage{amssymb}
\usepackage{graphicx}%
\providecommand{\U}[1]{\protect\rule{.1in}{.1in}}
\newtheorem{theorem}{Theorem}
\theoremstyle{plain}

\newtheorem{definition}{Definition}

\newtheorem{lemma}{Lemma}

\newtheorem{remark}{Remark}

\numberwithin{equation}{section}
\begin{document}
\title[Approximation by Random Polynomials]{Approximation of Random Functions by Random Polynomials in the
Framework of Choquet's Theory of Integration}
\author{Sorin G. Gal$^{*}$}
\address{Department of Mathematics and Computer Science\\
University of Oradea\\
University\ Street No. 1, Oradea, 410087, Romania}
\email{galso@uoradea.ro, galsorin23@gmail.com}
\author{Constantin P. Niculescu}
\address{Department of Mathematics, University of Craiova, Craiova 200585, Romania}
\email{constantin.p.niculescu@gmail.com}
\thanks{$^{*}$Corresponding author: Sorin G. Gal. E-mail: galso@uoradea.ro}
\date{September 29, 2020}
\dedicatory{ }
\begin{abstract}
Given a submodular capacity space, we prove the uniform convergence in
capacity and also the uniform convergence in the Choquet-mean of order $p\ge1$
with a quantitative estimate, of the multivariate Bernstein polynomials
associated to a random function. Applications to quantitative estimates
concerning the uniform convergence in capacity in the univariate case are given.
\end{abstract}
\maketitle

\section{Introduction}

In this paper we extend some old and new results on the approximation of
random functions by Bernstein random polynomials to the framework of Choquet's
theory of integrability. As is well known, these polynomials are among the
most studied and the most interesting polynomials used in the probabilistic
framework of approximation theory. We mention here the classical book of
Lorentz \cite{Lor} and the papers of Onicescu and Istr\u{a}\c{t}escu
\cite{Onic-2}, \cite{Onic-3}, Cenu\c{s}\u{a} and S\u{a}cuiu \cite{Cen-2}, Gal
\cite{Gal-Random-1}, \cite{Gal-Random-2}, and Gal and Villena
\cite{Gal-Random-3}. In the very recent papers of Adell and
C\'{a}rdenas-Morales \cite{Adell_2}, Sun and Wu \cite{Sun}, Wu, Sun and Ma
\cite{Wu_1} and Wu and Zhou \cite{Wu_2}, quantitative estimates for
approximation in probability of deterministic functions by random Bernstein
polynomials were obtained.

The papers cited above have motivated us to study the extension of the
approximation properties of random Bernstein polynomials in the much more
general framework provided by capacity spaces and the Choquet integral. Unlike
the case of probability measures, the capacities are nonadditive set
functions, and precisely the lack of additivity makes them useful in risk
theory (especially in decision making under risk and uncertainty). See
F\"{o}llmer and Schied \cite{FS} and Grabisch \cite{Gr}.

In Section 2 we present preliminaries on capacities and Choquet integral.
Section 3 is devoted to a description of various concepts of continuity of
random functions and of the convergence of sequences of random functions in
the setting of Choquet integral. Section 4 deals with approximation results by
random Bernstein polynomials of several variables in the framework of
capacities and Choquet integral. Our main results are Theorem \ref{thm1},
devoted to the approximation in the Choquet-mean of order $p\in\lbrack
1,\infty),$ and Theorem \ref{thm2}, devoted to the uniform approximation in
capacity by sequences of multivariate random Bernstein polynomials. In the
probabilistic case (and for functions of one real variable and $p=1$ for
Theorem \ref{thm1}), these results were previously proved respectively in
Cenu\c{s}\u{a} and S\u{a}cuiu \cite{Cen-2} and Onicescu and Istr\u{a}\c{t}escu
\cite{Onic-2}. In section 5 we obtain quantitative estimates for the
approximation in capacity by univariate Bernstein-type random polynomials,
generalizing recent results due to Adell and C\'{a}rdenas-Morales
\cite{Adell_2}, Sun and Wu \cite{Sun}, Wu, Sun and Ma \cite{Wu_1} and Wu and
Zhou \cite{Wu_2}, who considered only the framework of probability measures
and of deterministic functions.

\section{Preliminaries on capacities and Choquet integral}

For the convenience of the reader we will briefly recall some basic facts
concerning Choquet's theory of integrability with respect to a nondecreasing
set function (not necessarily additive). Full details are to be found in the
books of Denneberg \cite{Denn} and Grabisch \cite{Gr}.

Let $(\Omega,\mathcal{A})$\ be an arbitrarily fixed measurable space, that is,
a nonempty abstract set $\Omega$ endowed with a $\sigma$-algebra
${\mathcal{A}}$ of subsets of $\Omega.$

\begin{definition}
\label{def1} A set function $\mu:{\mathcal{A}}\rightarrow\mathbb{R}_{+}$ is
called a capacity if it verifies the following two conditions:

$(a)$ $\mu(\emptyset)=0$ and $\mu(\Omega)=1;$

$(b)~\mu(A)\leq\mu(B)$ for all $A,B\in{\mathcal{A}}$, with $A\subset B$.

A capacity $\mu$ is called subadditive if%
\[
\mu(A\bigcup B)\leq\mu(A)+\mu(B)
\]
and submodular \emph{(}or strongly subadditive\emph{)} if
\[
\mu(A\bigcup B)+\mu(A\bigcap B)\leq\mu(A)+\mu(B)\text{\quad}%
\]
for all $A,B\in{\mathcal{A}}.$

A capacity $\mu$ is called continuous from below (or lower continuous) if
\[
\lim_{n\to\infty}\mu(A_{n})=\mu\left(  \cup_{n=1}^{\infty}A_{n}\right)
\]
for every nondecreasing sequence $(A_{n})_{n}$ of sets in $\mathcal{A}$ such
that $\cup_{n=1}^{\infty}A_{n}\in\mathcal{A}$.

Analogously, a capacity $\mu$ is called continuous from above (or upper
continuous) if $\lim_{n\rightarrow\infty}\mu(A_{n})=\mu\left(  \cap
_{n=1}^{\infty}A_{n}\right)  $ for every nonincreasing sequence $(A_{n})_{n}$
of sets in $\mathcal{A}$.
\end{definition}

A simple way to construct nontrivial examples of submodular capacities is to
start with a probability measure $P:\mathcal{A\rightarrow}[0,1]$ and to
consider any nondecreasing concave function $u;[0,1]\rightarrow\lbrack0,1]$
such that $u(0)=0$ and $u(1)=1;$ for example, one may chose $u(t)=t^{a}$ with
$0<\alpha<1.$Then $\mu=u(P)$ is a submodular capacity on the $\sigma$-algebra
$\mathcal{A}$, called a \emph{distorted probability}.

Another particular class of capacities is given by the so-called measures of
possibilities (or maxitive set functions), defined as follows (see, e.g.,
\cite{Zadeh}, or \cite{Dubois}, Chapter 1, or \cite{Cooman}).

\begin{definition}
\label{def11} A set function $\mu:{\mathcal{A}}\rightarrow\mathbb{R}_{+}$ is
called possibility (or maxitive) measure, if it verifies the following axioms :

$(a)$ $\mu(\emptyset)=0$, $\mu(\Omega)=1$ and $\mu(\bigcup_{i\in I}A_{i}%
)=\sup\{\mu(A_{i}) ; i\in I\}$ for all $A_{i}\in\Omega$, and any $I$ family of indices.

$(b)$ A possibility distribution $($on $\Omega)$, is a function $\lambda
:\Omega\rightarrow\lbrack0,1]$, such that $\sup\{\lambda(s);s\in\Omega\}=1$.

According to, e.g., \cite{Dubois}, any possibility distribution $\lambda$ on
$\Omega$, induces a possibility measure $\mu$ given by the formula
$\mu_{\lambda}(A)=\sup\{\lambda(s) ; s\in A\}$, for all $A\subset\Omega$.
\end{definition}

Any possibility measure $\mu$ is monotone and submodular. Indeed, while the
monotonicity is immediate from the axiom $\mu(A\cup B)=\max\{\mu(A), \mu
(B)\}$, the submodularity is immediate from the property $\mu(A \cap B)
\le\min\{\mu(A), \mu(B)\}$.

The \emph{capacity spaces} (that is, the triplets $(\Omega,\mathcal{A},\mu),$
where $\Omega$ is a nonempty abstract set endowed with a $\sigma$-algebra
${\mathcal{A}}$ of subsets of $\Omega$ and $\mu:\mathcal{A}\rightarrow
\mathbb{R}_{+}$ is a capacity) represent a generalization of the classical
concept of probability space.

To a capacity space $(\Omega,\mathcal{A},\mu)$ one can attach several spaces
of functions, starting with the space $L^{0}(\Omega,\mathcal{A},\mu)$ of all
random variables $f:\Omega\rightarrow\mathbb{R}$ (that is, of all functions
$f$ verifying the condition of ${\mathcal{A}}$-measurability, $f^{-1}
(A)\in{\mathcal{A}}$ for every Borel subset $A\subset\mathbb{R}$). At the end
of this section, the analogs of the classical Lebesgue spaces $L^{p}%
(\Omega,\mathcal{A},\mu)$ (for $p\geq1)$ will be presented (under the
requirement that the capacity $\mu$ is submodular).

The key ingredient is the integrability of random variables $f\in L^{0}%
(\Omega,\mathcal{A},\mu)$ with respect to the capacity $\mu$.\qquad\qquad

\begin{definition}
\label{def2}The Choquet integral of a random variable $f:\Omega\rightarrow
\mathbb{R}$ on a set $A\in\mathcal{A}$ is defined by the formula
\begin{align}
(\operatorname*{C})\int_{A}f\mathrm{d}\mu &  =\int_{0}^{+\infty}\mu\left(
\{\omega\in\Omega:f(\omega)>t\}\cap A\right)  \mathrm{d}t\nonumber\\
&  +\int_{-\infty}^{0}\left[  \mu\left(  \{\omega\in\Omega:f(\omega)>t\}\cap
A\right)  -\mu(A)\right]  \mathrm{d}t, \label{Ch}%
\end{align}
where the integrals in the right hand side are generalized Riemann integrals.

If $(C)\int_{A}f\mathrm{d}\mu$ exists in $\mathbb{R}$, then $f$ is called
Choquet integrable on $A$.
\end{definition}

Notice that if $f\geq0$, then the last integral in the formula (\ref{Ch}) is 0.

The Choquet integral agrees with the Lebesgue integral in the case of
probabilistic measures. See Denneberg \cite{Denn}, p. 62.

The next remark summarizes the basic properties of the Choquet integral:

\begin{remark}
\label{rem1}$(a)$ If $f,g\in L^{0}(\Omega,\mathcal{A},\mu)$ are Choquet
integrable on $A$, then
\begin{gather*}
f\geq0\text{ implies }(\operatorname*{C})\int_{A}f\mathrm{d}\mu\geq0\text{
\quad\emph{(}positivity\emph{)}}\\
f\leq g\text{ implies }\left(  \operatorname*{C}\right)  \int_{A}%
f\mathrm{d}\mu\leq\left(  \operatorname*{C}\right)  \int_{A}g\mathrm{d}%
\mu\text{ \quad\emph{(}monotonicity\emph{)}}\\
\left(  \operatorname*{C}\right)  \int_{A}af\mathrm{d}\mu=a\cdot\left(
\operatorname*{C}\right)  \int_{A}f\mathrm{d}\mu\text{ for all }a\geq0\text{
\quad\emph{(}positive\emph{ }homogeneity\emph{)}}\\
\left(  \operatorname*{C}\right)  \int_{A}1\cdot\mathrm{d}\mu(t)=\mu
(A)\text{\quad\emph{(}calibration\emph{).}}%
\end{gather*}
$(b)$ In general, the Choquet integral is not additive but, if $f$ and $g$ are
comonotonic \emph{(}that is, $(f(\omega)-f(\omega^{\prime}))\cdot
(g(\omega)-g(\omega^{\prime}))\geq0$, for all $\omega,\omega^{\prime}\in
A$\emph{), }then
\[
\left(  \operatorname*{C}\right)  \int_{A}(f+g)\mathrm{d}\mu=\left(
\operatorname*{C}\right)  \int_{A}f\mathrm{d}\mu+\left(  \operatorname*{C}%
\right)  \int_{A}g\mathrm{d}\mu.
\]
An immediate consequence is the property of translation invariance,
\[
\left(  \operatorname*{C}\right)  \int_{A}(f+c)\mathrm{d}\mu=\left(
\operatorname*{C}\right)  \int_{A}f\mathrm{d}\mu+c\cdot\mu(A)
\]
for all $c\in\mathbb{R}$ and $f$ integrable on $A.$

($c)$ If $\mu$ is a subadditive capacity and $f$ is nonnegative and Choquet
integrable on the sets $A$ and $B$, then
\[
(\operatorname*{C})\int_{A\cup B}f\mathrm{d}\mu\leq(\operatorname*{C})\int
_{A}f\mathrm{d}\mu+(\operatorname*{C})\int_{B}f\mathrm{d}\mu.
\]
For $(a)$ and $(b)$ see Denneberg\emph{ \cite{Denn}, }Proposition\emph{ }%
$5.1$\emph{, }p\emph{. }$64;$\emph{ the assertion }$(c)$\emph{ follows in a
straightforward way from the definition of the Choquet integral.}
\end{remark}

\begin{remark}
\label{rem4}\emph{(The Subadditivity Theorem) }If $\mu$ is a submodular
capacity, then\ the associated Choquet integral is subadditive, that is,%
\[
\left(  \operatorname*{C}\right)  \int_{A}(f+g)\mathrm{d}\mu\leq\left(
\operatorname*{C}\right)  \int_{A}f\mathrm{d}\mu+\left(  \operatorname*{C}%
\right)  \int_{A}g\mathrm{d}\mu
\]
for all functions $f$ and $g$ integrable on $A.$ See\emph{ \cite{Denn},
}Theorem\emph{ }$6.3$\emph{, }p\emph{. }$75$\emph{. }In addition, the
following two integral analogs of the modulus inequality hold true:
\[
|(\operatorname*{C})\int_{A}f\mathrm{d}\mu|\leq(\operatorname*{C})\int
_{A}|f|\mathrm{d}\mu
\]
and
\[
|(\operatorname*{C})\int_{A}f\mathrm{d}\mu-(\operatorname*{C})\int
_{A}g\mathrm{d}\mu|\leq(\operatorname*{C})\int_{A}|f-g|\mathrm{d}\mu;
\]
the last assertion is covered by Corollary $6.6$, p. $82$, in
\emph{\cite{Denn}.}
\end{remark}

The analogs of the Lebesgue spaces in the context of capacities can be
introduced for $1\leq p<+\infty$ via the formulas
\[
{\mathcal{L}}^{p}(\Omega,\mathcal{A},\mu)=\{f:f\in L^{0}(\Omega,\mathcal{A}%
,\mu)\text{ and }(C)\int_{\Omega}|f(\omega)|^{p}\mathrm{d}\mu<+\infty\}.
\]
When $\mu$ is a subadditive capacity (in particular, when $\mu$ is
submodular), the functionals $\Vert\cdot\Vert_{{\mathcal{L}}^{p}%
(\Omega,\mathcal{A},\mu)}$ defined by the formula
\[
\Vert f\Vert_{{\mathcal{L}}^{p}(\Omega,\mathcal{A},\mu)}=\left(
(C)\int_{\Omega}|f(\omega)|^{p}\mathrm{d}\mu\right)  ^{1/p}%
\]
satisfy the triangle inequality too (see, e.g., \cite{Cerda}, Theorem 2, p. 5,
or \cite{Denn}, Proposition 9.4, pp. 109-110).

Under the stronger hypothesis that $\mu$ is a submodular capacity, the
quotient space
\[
L^{p}(\Omega,\mathcal{A},\mu)={\mathcal{L}}^{p}(\Omega,\mathcal{A}%
,\mu)/{\mathcal{N}}_{p},
\]
where
\[
{\mathcal{N}}_{p}=\{f\in{\mathcal{L}}^{p}(\Omega,\mathcal{A},\mu):\left(
(C)\int_{\Omega}|f(\omega)|^{p}\mathrm{d}\mu\right)  ^{1/p}=0\},
\]
becomes a normed vector space relative to the norm
\[
\Vert f\Vert_{L^{p}(\Omega,\mathcal{A},\mu)}=\left(  (C)\int_{\Omega}%
|f(\omega)|^{p}\mathrm{d}\mu\right)  ^{1/p}.
\]
See \cite{Denn}, Proposition 9.4, p. 109, for $p=1$ and ibidem p. 115 for
arbitrary $p\geq1$.

$L^{p}(\Omega,\mathcal{A},\mu)$ is a Banach space when $\mu$ is not only
submodular, but also lower continuous in the sense that
\[
\lim_{n\rightarrow\infty}\mu(A_{n})=\mu(\cup_{n=1}^{\infty}A_{n})
\]
for every nondecreasing sequence $(A_{n})_{n}$ of sets in $\mathcal{A}$. See
\cite{Denn}, Proposition 9.5, p. 111 and the comment at page 115. Under the
same hypotheses, $h\in{\mathcal{N}}_{p}$ if and only if
\[
h=0\text{\quad}\mu\text{-}a.e.,
\]
meaning the existence of a set $N\subset\Omega$ such that
\[
\mu^{\ast}(N)=\inf\left\{  \mu(A):A\in\mathcal{A},\text{ }A\supset N\text{
}\right\}  =0
\]
and $h(\omega)=0$ for all $\omega\in\Omega\setminus N.$ See \cite{Denn}, p.
107, Corollary 9.2, and the comments at pp. 107-108.

\section{Continuity of Random functions associated to a capacity space}

Given a capacity space $(\Omega,\mathcal{A},\mu)$ and a subset $D$ of the
Euclidean space $\mathbb{R}^{N},$ we will refer to the functions
$F:D\rightarrow L^{0}(\Omega,\mathcal{A},\mu)$ as \emph{random functions}. It
is also usual to interpret $F$ as a \emph{stochastic process} $F:D\times
\Omega\rightarrow\mathbb{R},$ $F(\mathbf{x},\omega)=F(\mathbf{x})(\omega).$
For fixed $\omega$, $F(\mathbf{x},\omega)$ is a deterministic function of
$\mathbf{x}$, called a \emph{sample function.}

Following the case of probabilistic spaces one can consider several kinds of
continuity, of interest for us being the following ones.

\begin{enumerate}
\item[-] A random function $F$ is \emph{continuous in capacity} at the point
$\mathbf{x}_{0}\in D$, if $\mathbf{x}\rightarrow\mathbf{x}_{0}$ implies
$F(\mathbf{x})\rightarrow F(\mathbf{x}_{0})$ in capacity, that is, for every
$\varepsilon>0$ and $\eta>0$ there exists $\delta=\delta(\varepsilon
,\eta,\mathbf{x}_{0})>0$ such that
\[
\mu(\{\omega\in\Omega:|F(\mathbf{x},\omega)-F(\mathbf{x}_{0},\omega
)|\geq\varepsilon\})<\eta
\]
whenever $\mathbf{x}\in E$ and $\Vert\mathbf{x}-\mathbf{x}_{0}\Vert<\delta.$

\item[-] A random function $F$ is called \emph{uniformly continuous in
capacity}, if for every $\varepsilon>0$ and $\eta>0$ there exists
$\delta=\delta(\varepsilon,\eta)>0$, such that%
\[
\mu(\{\omega\in\Omega:|F(\mathbf{x}^{\prime},\omega)-F(\mathbf{x}%
^{\prime\prime},\omega)|\geq\varepsilon\})<\eta
\]
whenever $\mathbf{x}^{\prime},\mathbf{x}^{\prime\prime}\in D$ and
$\Vert\mathbf{x}^{\prime}-\mathbf{x}^{\prime\prime}\Vert<\delta.$
\end{enumerate}

When $F$ takes values in a space $L^{p}(\Omega,\mathcal{A},\mu)$ (for some
$p\in\lbrack1,\infty)),$ then one can speak of its continuity in the
Choquet-mean of order $p.$

\begin{enumerate}
\item[-] A random function $F$ is called \emph{continuous in the Choquet-mean
of order} $p$ at the point $\mathbf{x}_{0}\in D$, if for every $\varepsilon>0$
there exists $\delta=\delta(\varepsilon,\mathbf{x}_{0})>0$, such that for all
$\mathbf{x}\in E$ with $\Vert\mathbf{x}-\mathbf{x}_{0}\Vert<\delta$, we have
\[
(C)\int_{\Omega}|F(\mathbf{x},\omega)-F(\mathbf{x}_{0},\omega)|^{p}%
\mathrm{d}\mu<\varepsilon.
\]

\item[-] A random function $F$ is called \emph{uniformly continuous in
Choquet-mean of order} $p$ if for every $\varepsilon>0$ there exists
$\delta(\varepsilon)>0$, such that for all $\mathbf{x}^{\prime},\mathbf{x}%
^{\prime\prime}\in E$ with $\Vert\mathbf{x}^{\prime}-\mathbf{x}^{\prime\prime
}\Vert<\delta(\varepsilon)$ we have
\[
(C)\int_{\Omega}|F(\mathbf{x}^{\prime},\omega)-F(\mathbf{x}^{\prime\prime
},\omega)|^{p}\mathrm{d}\mu<\varepsilon.
\]
In the next section we will be interested in the approximation of random
functions by random Bernstein polynomials. The notions of approximation in
capacity and approximation in Choquet-mean are defined as follows:

\item[-] A sequence $(F_{n})_{n}$ of random functions \emph{converges in
capacity} to the random function $F$ at $\mathbf{x}\in D$, if for every
$\varepsilon,\eta>0$, there exists $N(\varepsilon,\eta,\mathbf{x}%
)\in\mathbb{N}$ such that for all $n\geq N(\varepsilon,\eta,\mathbf{x})$ we
have
\[
\mu(\{\omega\in\Omega:|F_{n}(\mathbf{x},\omega)-F(\mathbf{x},\omega
)|\geq\varepsilon\})<\eta.
\]
If $N(\varepsilon,\eta,\mathbf{x})$ does not depend on $\mathbf{x}$, then we
say that $(F_{n})_{n}$ \emph{converges} \emph{uniformly in capacity} to $F$.

\item[-] A sequence $(F_{n})_{n}$ of random functions \emph{converges in
Choquet-mean of order} $p$ to the random function $F$ if for every
$\varepsilon>0$ and $\mathbf{x}\in D$, there exists $N(\varepsilon
,\mathbf{x})\in\mathbb{N}$ such that for all $n\geq N(\varepsilon,\mathbf{x})$
we have
\[
(C)\int_{\Omega}|F_{n}(\mathbf{x},\omega)-F(\mathbf{x},\omega)|^{p}%
\mathrm{d}\mu<\varepsilon.
\]
If $N(\varepsilon,\mathbf{x})$ does not depend on $\mathbf{x}$, then we will
say that $(F_{n})_{n}$ \emph{converges uniformly} \emph{to} $F$ \emph{in the
Choquet-mean of order} $p.$

\item[-] For $1\leq p<+\infty$ and $\delta_{i}\geq0$, $i=1,...,N$, the
multivariate \emph{Choquet $L^{p}$-modulus of continuity} of $F$ will be
defined by
\[
\Gamma(f:\delta_{1},...,\delta_{N})_{p}%
\]%
\[
=\left(  \sup_{|t_{i}-s_{i}|\leq\delta_{i},i=1,...N}(C)\int_{\Omega}%
|F(t_{1},...,t_{N},\omega)-F(s_{1},...,s_{N},\omega)|^{p}d\mu(\omega)\right)
^{1/p}.
\]

\item[-] A sequence $(F_{n})_{n}$ of random functions \emph{converges in
distribution with respect to the capacity} $\mu$, to the random function $F$
at $t_{0}\in D$, if
\[
\lim_{n\to\infty}D_{F_{n}}(t_{0}, x)=D_{F}(t_{0}, x),
\]
at each point $x\in\mathbb{R}$ where its distribution function $D_{F}(t_{0},
x)=\mu(\{\omega\in\Omega: F(t_{0}, \omega)\leq x\})$, is continuous as
function of $x$.

If the limit takes place uniformly with respect to $t_{0}\in D$, then we say
that the sequence of random functions $(F_{n})_{n}$ \emph{converges}
\emph{uniformly in distribution with respect to the capacity} $\mu$, to $F$.
\end{enumerate}

\begin{remark}
\label{distrib} A big source of convergence in distribution with respect to a
capacity is provided by convergence in capacity. Indeed, let $\mu$ be a
subadditive capacity. Since by Proposition 8.5, p. 98 in \cite{Denn}, if
$(f_{n})_{n}$ is a sequence of random variables converging in capacity to $f$,
then $(f_{n})_{n}$ converges in distribution to $f$ with respect to the same
capacity, it easily follows that if a sequence $(F_{n})_{n}$ of random
functions \emph{converges in capacity} to the random function $F$ at $t_{0}\in
D$, then it \emph{converges in distribution with respect to capacity}, to the
random function $F$ at $t_{0}\in D$.
\end{remark}

An important property of the Choquet $L^{p}$-modulus of continuity used in
approximation is the following one, stated and proved here only for simplicity
for two variables.

\begin{theorem}
\label{thm0}Let $1\leq p<+\infty$ and $N\in\mathbb{N}$. If $\mu$ is a
submodular capacity, then
\[
\Gamma(F:\alpha_{1}\cdot\gamma_{1},...,\alpha_{N}\cdot\gamma_{N})_{p}%
\leq(1+\alpha_{1}+...+\alpha_{N})\Gamma(F:\gamma_{1},...,\gamma_{N})_{p},
\]
for all $\alpha_{i},\gamma_{i}\geq0$, $i=1,...,N$.
\end{theorem}

\begin{proof}
For simplicity, we give the proof only for $N=2$, but the proof in the general
case for $N$ is similar. We start with the inequality
\begin{equation}
\Gamma(F:\delta_{1}+\delta_{2},\eta_{1}+\eta_{2})_{p}\leq\Gamma(F:\delta
_{1},\eta_{1})_{p}+\Gamma(F:\delta_{2},\eta_{2})_{p}. \label{aditiv}%
\end{equation}
Indeed, let $t_{1},r_{1},s_{1}$ with $|t_{1}-s_{1}|\leq\delta_{1}+\delta_{2}$,
$|t_{1}-r_{1}|\leq\delta_{1}$, $|r_{1}-s_{1}|\leq\delta_{2}$ and $t_{2}%
,r_{2},s_{2}$ with $|t_{2}-s_{2}|\leq\eta_{1}+\eta_{2}$, $|t_{2}-r_{2}%
|\leq\eta_{1}$, $|r_{2}-s_{2}|\leq\eta_{2}$.

Since $\mu$ is submodular, the Minkowski inequality is available in the vector
space ${\mathcal{L}}(\Omega,{\mathcal{A}},\mu)$; see Theorem 2, p. 5 in
\cite{Cerda}, or Proposition 9.4, p. 109-110 in \cite{Denn}. Therefore
\begin{multline*}
\left(  (C)\int_{\Omega}|F(t_{1},t_{2},\omega)-F(s_{1},s_{2},\omega)|^{p}%
d\mu(\omega)\right)  ^{1/p}\\
\leq\left(  (C)\int_{\Omega}|F(t_{1},t_{2},\omega)-F(r_{1},r_{2},\omega
)|^{p}d\mu(\omega)\right)  ^{1/p}\\
+\left(  (C)\int_{\Omega}|F(r_{1},r_{2},\omega)-F(s_{1},s_{2},\omega)|^{p}%
d\mu(\omega)\right)  ^{1/p}.
\end{multline*}
Passing now to the corresponding suprema, first in the right-hand side and
then in the left-hand side, we are led to (\ref{aditiv}). As a consequence
\[
\Gamma(F:n\delta,m\eta)_{p}\leq\max\{n,m\}\Gamma(F:\delta,\eta)_{p}%
\]
and taking into account that $\alpha<[\alpha]+1,~\beta<[\beta]+1$ and
\[
\max\{[\alpha]+1,[\beta]+1\}=\max\{[\alpha],[\beta]\}+1<\alpha+\beta+1,
\]
one easily obtain the inequality in the statement of Theorem \ref{thm0}.
\end{proof}

\section{Approximation via random Bernstein polynomials}

The approximation of random functions defined on a compact $N$-dimensional
interval in $\mathbb{R}^{N}$ (that is, on a product of $N$ compact intervals
of $\mathbb{R}$) can be easily reduced (via an affine transformation) to the
particular case where the domain is the $N$-dimensional unit cube $[0,1]^{N}.$
In this context it is important to study the approximation of random functions
$F:[0,1]^{N}\rightarrow L^{0}(\Omega,\mathcal{A},\mu)$ via the associated
random Bernstein polynomials,%
\begin{multline*}
B_{n_{1},...,n_{N}}(F)(x_{1},...,x_{N},\omega)\\
=\sum_{k_{1}=0}^{n_{1}}...\sum_{k_{N}=0}^{n_{N}}p_{k_{1},n_{1}}(x_{1})\cdots
p_{k_{N},n_{N}}(x_{N})\cdot F\left(  \frac{k_{1}}{n_{1}},...,\frac{k_{N}%
}{n_{N}},\omega\right)  ,
\end{multline*}
where $p_{k_{j},n_{j}}(x_{j})={\binom{n_{j}}{k_{j}}}x_{j}^{k_{j}}%
(1-x_{j})^{n_{j}-k_{j}}$, $k_{j}\in\{0,...,n_{j}\},$ $n_{j}\in\mathbb{N}$ and
$x_{j}\in\lbrack0,1]$ for $j=1,...N$.

Recall that the classical Bernstein polynomials attached to a function
$f:[0,1]\rightarrow\mathbb{R}$ are defined by the formula
\[
B_{n}(f)(x)=\sum_{k=0}^{n}f\left(  \frac{k}{n}\right)  p_{k,n}(x),\text{\quad
}x\in\lbrack0,1],~n\in\mathbb{N},
\]
and their main feature is the estimate
\begin{equation}
\sup_{x}\Vert B_{n}(f)(x)-f(x)\Vert\leq c\cdot\omega_{1}\left(  f;\frac
{1}{\sqrt{n}}\right)  , \label{Sik}%
\end{equation}
where $c=\frac{4306+837\sqrt{6}}{5832}=1,089....$ is the optimal Sikkema
constant and
\[
\omega_{1}(f;\delta)=\sup\{|f(x)-f(y)|:x,y\in\lbrack0,1],|x-y|\leq\delta\}
\]
is the usual modulus of continuity. See \cite{Sikkema}.

These polynomials also have a number of nice properties related to shape
preservation, that make them useful to computer aided geometric design.
Details are available in the book of Gal \cite{Gal2008}.

The approximation of random functions by random Bernstein polynomials will be
discussed in the context of \emph{submodular capacity spaces} $(\Omega
,\mathcal{A},\mu),$ that is, when the capacity $\mu$ under attention is
submodular. We start with the case of approximation in the Choquet-mean of
order $p\in\lbrack1,\infty)$.

\begin{theorem}
\label{thm1} Suppose that $(\Omega,\mathcal{A},\mu)$ is a submodular capacity
space and
\[
F:[0,1]^{N}\rightarrow L^{p}(\Omega,\mathcal{A},\mu)
\]
is a random function. Then for all $x_{1},x_{2},...,x_{N}\in\lbrack0,1]$ and
$n_{1},n_{2},...,n_{N}\in\mathbb{N}$, the following quantitative estimate
holds
\begin{multline*}
\left[  (C)\int_{\Omega}|F(x_{1},x_{2},...x_{N},\omega)-B_{n_{1},n_{2}%
}(F)(x_{1},x_{2},...,x_{N},\omega)|^{p}\mathrm{d}\mu\right]  ^{1/p}\\
\leq\lbrack C_{p}]^{1/p}\cdot\Gamma\left(  F;\frac{1}{\sqrt{n_{1}}},\frac
{1}{\sqrt{n_{2}}},...,\frac{1}{\sqrt{n_{N}}}\right)  _{p},
\end{multline*}
where $C_{p}$ is independent of $n_{1}$, $n_{2},...,n_{N}$ and $x_{1}%
,x_{2},...,x_{N}$.

If $F$ is continuous in the Choquet-mean of order $p$ at each $\mathbf{x}%
\in\lbrack0,1]^{N}$, then the sequence of random Bernstein polynomials
$(B_{n_{1},...,n_{N}}(F))_{n_{1},...,n_{N}}$ converges uniformly to $F$ in the
Choquet-mean of order $p$ as $\min\left\{  n_{1},...,n_{N}\right\}
\rightarrow\infty.$
\end{theorem}

\begin{proof}
For simplicity, we will give all the details of the proof in the case $N=2$
(the general case being similar). Taking into account the identity
\[
\sum_{k_{1}=0}^{n_{1}}\sum_{k_{2}=0}^{n_{2}}p_{k_{1},n_{1}}(x_{1})\cdot
p_{k_{2},n_{2}}(x_{2})=1
\]
and the convexity of the function $x^{p}$ for $p\geq1$, we infer from Jensen's
inequality that
\begin{multline*}
|F(x_{1},x_{2},\omega)-B_{n_{1},n_{2}}(F)(x_{1},x_{2},\omega)|^{p}\\
\leq\left[  \sum_{k_{1}=0}^{n_{1}}\sum_{k_{2}=0}^{n_{2}}p_{k_{1},n_{1}}%
(x_{1})\cdot p_{k_{2},n_{2}}(x_{2})|F(x_{1},x_{2},\omega)-F(k_{1}/n_{1}%
,k_{2}/n_{2},\omega)|\right]  ^{p}\\
\leq\sum_{k_{1}=0}^{n_{1}}\sum_{k_{2}=0}^{n_{2}}p_{k_{1},n_{1}}(x_{1})\cdot
p_{k_{2},n_{2}}(x_{2})|F(x_{1},x_{2},\omega)-F(k_{1}/n_{1},k_{2}/n_{2}%
,\omega)|^{p}.
\end{multline*}
Integrating side by side and using Remark \ref{rem1}, (a) and (c) we arrive at
the estimate
\[
(C)\int_{\Omega}|F(x_{1},x_{2},,\omega)-B_{n_{1},n_{2}}(F)(x_{1},x_{2}%
,\omega)|^{p}\mathrm{d}\mu
\]%
\begin{equation}
\leq\sum_{k_{1}=0}^{n_{1}}\sum_{k_{2}=0}^{n_{2}}p_{k_{1},n_{1}}(x_{1}%
)p_{k_{2},n_{2}}(x_{2})(C)\int_{\Omega}|F(x_{1},x_{2},\omega)-F(k_{1}%
/n_{1},k_{2}/n_{2},\omega)|^{p}\mathrm{d}\mu. \label{elp}%
\end{equation}
Using the inequality (\ref{elp}) and then Theorem \ref{thm0}, we get
\begin{multline*}
(C)\int_{\Omega}|F(x_{1},x_{2},,\omega)-B_{n_{1},n_{2}}(F)(x_{1},x_{2}%
,\omega)|^{p}\mathrm{d}\mu\\
\leq\sum_{k_{1}=0}^{n_{1}}\sum_{k_{2}=0}^{n_{2}}p_{k_{1},n_{1}}(x_{1}%
)p_{k_{2},n_{2}}(x_{2})\\
\cdot\left[  \Gamma\left(  F;\frac{1}{\sqrt{n_{1}}}\cdot(\sqrt{n_{1}}%
|x_{1}-k_{1}/n_{1}|),\frac{1}{\sqrt{n_{2}}}\cdot(\sqrt{n_{2}}|x_{2}%
-k_{2}/n_{2}|)\right)  _{p}\right]  ^{p}\\
\leq\left[  \Gamma\left(  F;\frac{1}{\sqrt{n_{1}}},\frac{1}{\sqrt{n_{2}}%
}\right)  _{p}\right]  ^{p}\\
\cdot\sum_{k_{1}=0}^{n_{1}}\sum_{k_{2}=0}^{n_{2}}p_{k_{1},n_{1}}%
(x_{1})p_{k_{2},n_{2}}(x_{2})(1+\sqrt{n_{1}}|x_{1}-k_{1}/n_{1}|+\sqrt{n_{2}%
}|x_{2}-k_{2}/n_{2}|)^{p}.
\end{multline*}
But by the general estimate of the moments of Bernstein polynomials
\[
\sum_{k=0}^{n}p_{k,n}(x)[\sqrt{n}|x-k/n|]^{j}\leq2G(1+j/2),j=0,1,...,p,
\]
where with $G$ we have denoted the Gamma function \emph{(}see Theorem \emph{1}
in J. A. Adell, J. Bustamente and J. M. Quesada \emph{\cite{Adell})}, it is
immediate that
\[
\sum_{k_{1}=0}^{n_{1}}\sum_{k_{2}=0}^{n_{2}}p_{k_{1},n_{1}}(x_{1}%
)p_{k_{2},n_{2}}(x_{2})(1+\sqrt{n_{1}}|x_{1}-k_{1}/n_{1}|+\sqrt{n_{2}}%
|x_{2}-k_{2}/n_{2}|)^{p}\leq C_{p},
\]
where $C_{p}$ is independent of $n_{1}$, $n_{2}$ and $x_{1},x_{2}\in
\lbrack0,1]$. Concluding, we obtain
\begin{multline*}
\left[  (C)\int_{\Omega}|F(x_{1},x_{2},,\omega)-B_{n_{1},n_{2}}(F)(x_{1}%
,x_{2},\omega)|^{p}\mathrm{d}\mu\right]  ^{1/p}\\
\leq\lbrack C_{p}]^{1/p}\cdot\Gamma\left(  F;\frac{1}{\sqrt{n_{1}}},\frac
{1}{\sqrt{n_{2}}}\right)  _{p}.
\end{multline*}
On the other hand, we observe that the continuity of $F$ in the Choquet-mean
of order $p$ at each $x$ in the compact $[0,1]^{N}$, easily implies its
uniform continuity on $[0,1]^{N}$, which by the definition of the multivariate
\emph{Choquet $L^{p}$-modulus of continuity} of $F$, immediately implies that
$\lim_{\delta_{1},...,\delta_{N}\rightarrow0}\Gamma(F;\delta_{1}%
,...,\delta_{N})_{p}=0$. This implies the second part of the theorem too.
\end{proof}

\begin{remark}
The particular case of Theorem $\ref{thm1}$, when $\mu$ is a $\sigma$-additive
measure, $N=1$ and $p=2$, was previously proved by Ignatov, Mills and Tzankova
\cite{Ignatov} and Kamolov \cite{Kamolov}. Also, the second part of Theorem
$\ref{thm1}$, for $\mu$ a $\sigma$-additive measure, $N=1$ and $p=1$ was first
noticed by Cenu\c{s}\u{a} and S\u{a}cuiu \cite{Cen-2}.
\end{remark}

The next result deals with the approximation in capacity.

\begin{theorem}
\label{thm2} Suppose that $(\Omega,\mathcal{A},\mu)$ is a submodular capacity
space and
\[
F:[0,1]^{N}\rightarrow L^{0}(\Omega,\mathcal{A},\mu)
\]
is a random function which is continuous in capacity at each point
$\mathbf{x}\in\lbrack0,1]^{N}$ and verifies the boundedness condition%
\[
M=\sup_{\mathbf{x}\in\lbrack0,1]^{N}}|F(\mathbf{x},\omega)|<+\infty
\]
for all $\omega\in\Omega,$ except possibly for a set $E$ of capacity zero.
Then the sequence of random Bernstein polynomials $(B_{n_{1},...,n_{N}%
}(F))_{n_{1},...,n_{N}}$ converges uniformly to $F$ in capacity as
$\min\left\{  n_{1},...,n_{N}\right\}  \rightarrow\infty.$
\end{theorem}

Combining Theorem \ref{thm1} and Theorem \ref{thm2}, it follows that the
random Bernstein polynomials also converge in distribution with respect to the
submodular capacity $\mu$.

\begin{proof}
For simplicity, we will detail the proof in the case $N=2$. As in the
classical case, let us consider the semi-metric
\[
d(F,G)=\sup_{\mathbf{x}\in\lbrack0,1]^{2}}(C)\int_{\Omega}\frac{|F(\mathbf{x}%
,\omega)-G(\mathbf{x},\omega)|}{1+|F(\mathbf{x},\omega)-G(\mathbf{x},\omega
)|}\mathrm{d}\mu.
\]
Indeed, from the properties of the Choquet integral as mentioned in Remark
\ref{rem1}, $(a)$ and $(c)$, we easily get that $d$ satisfies the triangle inequality.

The fact that the convergence with respect to $d$ implies the uniform
convergence in capacity, is a direct consequence of Markov's inequality (for
the Choquet integral). Keeping fixed $\mathbf{x}\in\lbrack0,1]^{2}$ and
assuming that $H(\mathbf{x},\omega)$ is a nonnegative random variable, then
for each $a>0$ we have
\begin{multline*}
(C)\int_{\Omega}H(\mathbf{x},\omega)\mathrm{d}\mu\geq(C)\int_{\Omega
\cap\{\omega\in\Omega:H(\mathbf{x},\omega)\geq a\}}H(\mathbf{x},\omega
)\mathrm{d}\mu\\
\geq(C)\int_{\Omega\cap\{\omega\in\Omega::H(\mathbf{x},\omega)\geq
a\}}a\mathrm{d}\mu=a\cdot\mu(\{\omega\in\Omega:H(\mathbf{x},\omega)\geq a\}),
\end{multline*}
which is Markov's inequality. It can be generalized by considering a positive
and strictly increasing function $\varphi$ on $[0,+\infty)$. Indeed,
\begin{gather}
\mu(\{\omega\in\Omega:H(\mathbf{x},\omega)\geq a\})=\mu(\{\omega\in
\Omega:\varphi(H(\mathbf{x},\omega))\geq\varphi(a)\})\label{Markov}\\
\leq\frac{(C)\int_{\Omega}\varphi(H(\mathbf{x},\omega))\mathrm{d}\mu}%
{\varphi(a)}.\nonumber
\end{gather}
Choosing $\varphi(t)=\frac{t}{1+t}$ and $H(\mathbf{x},\omega)=\left\vert
F(\mathbf{x},\omega)-G(\mathbf{x},\omega)\right\vert $ in (\ref{Markov}), one
can easily see that the convergence in the metric $d$ implies the uniform
convergence in capacity.

Concerning the set $E$ in the hypothesis, let us notice that any random
variable $G(\mathbf{x},\omega)\geq0$ verifies%
\[
(C)\int_{E}G\mathrm{d}\mu=\int_{0}^{\infty}\mu\left(  \left\{  x:f(x)\geq
t\right\}  \cap E\right)  \mathrm{d}t=0.
\]
According to assertions (a) and (c) of Remark \ref{rem1},
\[
(C)\int_{\Omega}G\mathrm{d}\mu\leq(C)\int_{E}G\mathrm{d}\mu+(C)\int
_{\Omega\setminus E}G\mathrm{d}\mu=(C)\int_{\Omega\setminus E}G\mathrm{d}\mu
\]
and thus
\begin{equation}
(C)\int_{\Omega}G\mathrm{d}\mu=(C)\int_{\Omega\setminus E}G\mathrm{d}\mu.
\label{eqE}%
\end{equation}
Next, notice that due to the compactness of the $[0,1]^{2}$, the function $F$
is uniformly continuous in capacity. This can easily be done by reductio ad absurdum.

As a consequence, for $\varepsilon>0$ arbitrary fixed there exists
$\delta(\varepsilon)$ such that
\begin{equation}
\mu\left(  \{\omega\in\Omega:|F(x_{1}^{\prime},x_{2}^{\prime},\omega
)-F(x_{1}^{\prime\prime},x_{2}^{\prime\prime},\omega)|\geq\varepsilon
/2\}\right)  <\frac{\varepsilon}{2M} \label{equa}%
\end{equation}
for all $x_{1}^{\prime},x_{2}^{\prime},x_{1}^{\prime\prime},x_{2}%
^{\prime\prime}\in\lbrack0,1]$ with $\max\left\{  \left\vert x_{1}^{\prime
}-x_{1}^{\prime\prime}\right\vert ,\left\vert x_{2}^{\prime}-x_{2}%
^{\prime\prime}\right\vert \right\}  <\delta(\varepsilon).$

One can also choose an integer $N(\varepsilon)$ such that%
\[
\frac{M}{2n\delta^{2}}<\frac{\varepsilon}{2}\text{\quad for all }n\geq
N(\varepsilon).
\]
Fix an arbitrary pair of integers $n_{1},n_{2}\geq N(\varepsilon)$ and define
the sets
\[
I_{1}^{\prime}=\{0\leq k_{1}\leq n_{1}:|k_{1}/n_{1}-x_{1}|<\delta
(\varepsilon)\},~I_{1}^{\prime\prime}=\{0\leq k_{1}\leq n_{1}:|k_{1}%
/n_{1}-x_{1}|\geq\delta(\varepsilon)\},
\]%
\[
I_{2}^{\prime}=\{0\leq k_{2}\leq n_{2}:|k_{2}/n_{2}-x_{2}|<\delta
(\varepsilon)\},~I_{2}^{\prime\prime}=\{0\leq k_{2}\leq n_{2}:|k_{2}%
/n_{2}-x_{2}|\geq\delta(\varepsilon)\}.
\]
We will also need the following estimate,
\begin{equation}
\sum_{k_{j}\in I_{j}^{\prime\prime}}p_{k_{j},n_{j}}(x_{j})\leq\frac{1}%
{4n_{j}\delta^{2}}\quad\text{for }j=1,2; \label{Lorentz}%
\end{equation}
see inequality (7) in Lorentz \cite{Lor}, p. 6. Put
\begin{align*}
\Omega_{1}  &  =\{\omega\in\Omega:|F(x_{1},x_{2},\omega)-F(k_{1}/n_{1}%
,k_{2}/n_{2},\omega)|<\varepsilon/2\},\,\\
\Omega_{2}  &  =\{\omega\in\Omega:|F(x_{1},x_{2},\omega)-F(k_{1}/n_{1}%
,k_{2}/n_{2},\omega)|\geq\varepsilon/2\}.
\end{align*}
Then we obtain the following partition of $\Omega$
\begin{equation}
E,\ \Omega_{1}^{\prime}=\Omega_{1}\setminus E\text{ and }\Omega_{2}^{\prime
}=\Omega_{2}\setminus E. \label{eqpart}%
\end{equation}

Taking into account Remark \ref{rem1}, for all $x_{1},x_{2}\in\lbrack0,1]$ and
$n_{1},n_{2}\geq N(\varepsilon)$ we have
\begin{multline*}
d(F,B_{n_{1},n_{2}}(F))\leq\sup_{x_{1},x_{2}\in\lbrack0,1]}(C)\int_{\Omega
}|F(x_{1},x_{2},\omega)-B_{n_{1},n_{2}}(F)(x_{1},x_{2},\omega)|\mathrm{d}\mu\\
\leq\sup_{x+1,x_{2}\in\lbrack0,1]}\sum_{k_{1}=0}^{n_{1}}\sum_{k_{2}=0}^{n_{2}%
}p_{k_{1},n_{1}}(x_{1})p_{k_{2},n_{2}}(x_{2})\cdot(C)\int_{\Omega}|\Delta
F(x_{1},x_{2};k_{1}/n_{1},k_{2}/n_{2})|\mathrm{d}\mu,
\end{multline*}
where $\Delta F(x_{1},x_{2};k_{1}/n_{1},k_{2}/n_{2})=F(x_{1},x_{2}%
,\omega)-F(k_{1}/n_{1},k_{2}/n_{2},\omega)$.

The last sum can be written as
\begin{align*}
\sum_{k_{1}=0}^{n_{1}}\sum_{k_{2}=0}^{n_{2}}  &  =\sum_{k_{1}\in I_{1}%
^{\prime}}\sum_{k_{2}\in I_{2}^{\prime}}+\sum_{k_{1}\in I_{1}^{\prime}}%
\sum_{k_{2}\in I_{2}^{\prime\prime}}+\sum_{k_{1}\in I_{1}^{\prime\prime}}%
\sum_{k_{2}\in I_{2}^{\prime}}+\sum_{k_{1}\in I_{1}^{\prime\prime}}\sum
_{k_{2}\in I_{2}^{\prime\prime}}\\
&  :=A_{1}+A_{2}+A_{3}+A_{4}.
\end{align*}
Then, based on Remark 1 (c) and equation (\ref{eqE}), we have
\begin{align*}
A_{1}  &  \leq\sum_{k_{1}\in I_{1}^{\prime}}\sum_{k_{2}\in I_{2}^{\prime}%
}p_{k_{1},n_{1}}(x_{1})\cdot p_{k_{2},n_{2}}(x_{2})(C)\int_{\Omega_{1}}|\Delta
F(x_{1},x_{2};k_{1}/n_{1},k_{2}/n_{2})|\mathrm{d}\mu\\
&  +\sum_{k_{1}\in I_{1}^{\prime}}\sum_{k_{2}\in I_{2}^{\prime}}p_{k_{1}%
,n_{1}}(x_{1})\cdot p_{k_{2},n_{2}}(x_{2})(C)\int_{\Omega_{2}}|\Delta
F(x_{1},x_{2};k_{1}/n_{1},k_{2}/n_{2})|\mathrm{d}\mu\\
&  +\sum_{k_{1}\in I_{1}^{\prime}}\sum_{k_{2}\in I_{2}^{\prime}}p_{k_{1}%
,n_{1}}(x_{1})\cdot p_{k_{2},n_{2}}(x_{2})(C)\int_{E}|\Delta F(x_{1}%
,x_{2};k_{1}/n_{1},k_{2}/n_{2})|\mathrm{d}\mu\\
&  \leq\frac{\varepsilon}{2}+2M\cdot\mu(\Omega_{2}^{\prime})+0=\frac
{\varepsilon}{2}+\varepsilon=\frac{3\varepsilon}{2}.
\end{align*}
Next, taking into account the estimate (\ref{Lorentz}),
\[
A_{2}\leq\sum_{k_{1}\in I_{1}^{\prime}}\sum_{k_{2}\in I_{2}^{\prime\prime}%
}p_{k_{1},n_{1}}(x_{1})\cdot p_{k_{2},n_{2}}(x_{2})(C)\int_{\Omega_{1}}|\Delta
F(x_{1},x_{2};k_{1}/n_{1},k_{2}/n_{2})|\mathrm{d}\mu
\]%
\[
+\sum_{k_{1}\in I_{1}^{\prime}}\sum_{k_{2}\in I_{2}^{\prime\prime}}%
p_{k_{1},n_{1}}(x_{1})\cdot p_{k_{2},n_{2}}(x_{2})(C)\int_{\Omega_{2}}|\Delta
F(x_{1},x_{2};k_{1}/n_{1},k_{2}/n_{2})|\mathrm{d}\mu
\]%
\[
+\sum_{k_{1}\in I_{1}^{\prime}}\sum_{k_{2}\in I_{2}^{\prime\prime}}%
p_{k_{1},n_{1}}(x_{1})\cdot p_{k_{2},n_{2}}(x_{2})(C)\int_{E}|\Delta
F(x_{1},x_{2};k_{1}/n_{1},k_{2}/n_{2})|\mathrm{d}\mu
\]%
\[
\leq\frac{\varepsilon}{2}+2M\cdot\sum_{k_{2}\in I_{2}^{\prime\prime}}%
p_{k_{2},n_{2}}(x_{2})+0\leq\frac{\varepsilon}{2}+\frac{\varepsilon}%
{2}=\varepsilon.
\]
Reasoning for $A_{3}$ and $A_{4}$ analogously, we get $A_{3}\leq\varepsilon$
and $A_{4}\leq\varepsilon$.

Concluding, we obtain $d(F,B_{n_{1},n_{2}}(F))\leq5\varepsilon$, for all
$n_{1},n_{2}\geq N(\varepsilon)$.
\end{proof}

\begin{remark}
Theorem $\ref{thm2}$ extends a result proved by Onicescu and
Istr\u{a}\c{t}escu \cite{Onic-2} in the particular case when $\mu$ is a
$\sigma$-additive measure and $N=1$.
\end{remark}

In the special case when the capacity $\mu$ is a measure of possibility, the
condition of boundedness of $F$ in Theorem \ref{thm2} can be removed:

\begin{theorem}
\label{thm3} Suppose that $(\Omega,\mathcal{A},\mu)$ is a capacity space with
$\mu$ a measure of possibility and that $F:[0,1]^{N}\rightarrow L^{0}%
(\Omega,\mathcal{A},\mu)$ is a random function which is continuous in capacity
at each point $\mathbf{x}\in\lbrack0,1]^{N}$.

Then the sequence of random Bernstein polynomials $(B_{n_{1},...,n_{N}%
}(F))_{n_{1},...,n_{N}}$ converges uniformly to $F$ in capacity as
$\min\left\{  n_{1},...,n_{N}\right\}  \rightarrow\infty.$
\end{theorem}

\begin{proof}
The proof is done in three steps.

Step 1. We start by arguing (by reduction at absurdum) that $F$ is uniformly
continuous in capacity on $[0,1]^{N}$. Indeed, suppose that $F$ is not
uniformly continuous. Then there exist $\varepsilon_{0},\eta_{0}>0$ and two
sequences $x_{n},y_{n}\in\lbrack0,1]^{N}$, $n\in\mathbb{N}$, with $\Vert
x_{n}-y_{n}\Vert\leq\frac{1}{n}$, such that
\[
\mu(\{\omega\in\Omega:|F(x_{n},\omega)-F(y_{n},\omega)|\geq\varepsilon
_{0}\})\geq\eta_{0}\quad\text{for }n\in\mathbb{N}.
\]
It is clear that we can suppose that both sequences $x_{n},y_{n}$ converge to
the same $x_{0}\in\lbrack0,1]^{N}$ and since
\[
\varepsilon_{0}\leq|F(x_{n},\omega)-F(y_{n},\omega)|\leq|F(x_{n}%
,\omega)-F(x_{0},\omega)|+|F(x_{0},\omega)-F(y_{n},\omega)|,
\]
it follows that
\begin{multline*}
\{\omega\in\Omega:|F(x_{n},\omega)-F(y_{n},\omega)|\geq\varepsilon_{0}\}\\
\subset\{\omega\in\Omega:|F(x_{n},\omega)-F(x_{0},\omega)|\geq\varepsilon
_{0}/2\}\cup\{\omega\in\Omega;|F(x_{0},\omega)-F(y_{n},\omega)|\geq
\varepsilon_{0}/2\}.
\end{multline*}
Since $\mu$ is subadditive, we get
\begin{multline*}
0<\eta_{0}\leq\mu(\{\omega\in\Omega:|F(x_{n},\omega)-F(x_{0},\omega
)|\geq\varepsilon_{0}/2\})\\
+\mu(\{\omega\in\Omega:|F(x_{0},\omega)-F(y_{n},\omega)|\geq\varepsilon
_{0}/2\}),
\end{multline*}
which combined with the hypothesis that $F$ is continuous in capacity at
$x_{0}$, leads us to a contradiction.

Step 2. Let us define the stochastic modulus of continuity by the formula
\[
O(F;\delta,\omega)=\sup\{|F(x,\omega)-F(y,\omega)|:\Vert x-y\Vert\leq
\delta,~x,y\in\lbrack0,1]^{N}\},~\omega\in\Omega.
\]
From the classical quantitative approximation result for multivariate
Bernstein polynomials, it follows that for all $x\in\lbrack0,1]^{N}$,
$\omega\in\Omega$ and $n\in\mathbb{N}$, we have
\begin{equation}
|F(x,\omega)-B_{n}(F)(x,\omega)|\leq C\cdot O(F;\frac{1}{\sqrt{n}},\omega),
\label{Est_Bern}%
\end{equation}
where $n=\min\{n_{1},...,n_{N}\}$ and $C>0$ is independent on $x$, $\omega$
and $n$ (but not on $N$).

Now, since $F$ is uniformly continuous in capacity, we can prove that
$G_{n}(\omega)=O(F;\frac{1}{\sqrt{n}},\omega)$ converges in capacity to $0$ as
$n\rightarrow\infty$. Indeed, for every $\varepsilon,\eta>0$ there exist
$\delta=\delta(\varepsilon,\eta)>0$ such that for all $x,y\in\lbrack0,1]^{N}$
with $\Vert x-y\Vert\leq\delta$, we have $\mu(\{\omega\in\Omega:|F(x,\omega
)-F(y,\omega)|\geq\varepsilon\})<\eta$. Since $\varepsilon>0$ was arbitrary
chosen, without loss of generality, we can restate this last conclusion as
\[
\mu(\{\omega\in\Omega:|F(x,\omega)-F(y,\omega)|>\varepsilon\})<\eta.
\]
On the other hand,
\begin{multline*}
\{\omega\in\Omega:O(F;\delta,\omega)>\varepsilon\}\\
=\{\omega\in\Omega:\text{there~exist}~x,y\in\lbrack0,1]^{N}\text{ with }\Vert
x-y\Vert\leq\delta\text{ and }|F(x,\omega)-F(y,\omega)|>\varepsilon\}\\
\subset\bigcup_{\Vert x-y\Vert\leq\delta,~x,y\in\lbrack0,1]^{N}}\{\omega
\in\Omega:|F(x,\omega)-F(y,\omega)|>\varepsilon\}.
\end{multline*}
This easily implies
\begin{align*}
\mu(\{\omega &  \in\Omega:O(F;\delta,\omega)>\varepsilon\})\\
&  \leq\sup_{\Vert x-y\Vert\leq\delta,~x,y\in\lbrack0,1]^{N}}\mu(\{\omega
\in\Omega:|F(x,\omega)-F(y,\omega)|>\varepsilon\})<\eta.
\end{align*}
Now, if $n_{0}\in\mathbb{N}$ is chosen such that $\frac{1}{\sqrt{n}}<\delta$
for all $n\geq n_{0}$, we get
\[
\mu(\{\omega\in\Omega:G_{n}(\omega)>\varepsilon\})\leq\eta,
\]
for all $n\geq n_{0}$; we used here that $\mu$ is a measure of possibility.
Since $\eta>0$ was arbitrarily chosen, this easily implies that for any
$\varepsilon>0$, we have
\[
\lim_{n\rightarrow\infty}\mu(\{\omega\in\Omega:G_{n}(\omega)>\varepsilon
\})=0.
\]
Step 3. Finally, combining Step 1 and Step 2, we infer that $B_{n}%
(F)(x,\omega)$ converges uniformly in the capacity $\mu$ to $F$ on $[0,1]^{N}$.
\end{proof}

\section{Quantitative estimates for convergence in capacity}

This section is devoted to the proof of several quantitative estimates for the
approximation in capacity by univariate Bernstein-type random polynomials. Our
results were inspired by the recent papers of Adell and C\'{a}rdenas-Morales
\cite{Adell_2}, Sun and Wu \cite{Sun}, Wu, Sun and Ma \cite{Wu_1} and Wu and
Zhou \cite{Wu_2}, who considered only the framework of probability measures
and of deterministic functions.

In the definition of the classical univariate Bernstein polynomials, the
function $f:[0,1]\rightarrow\mathbb{R}$ is evaluated at the set of equally
spaced nodes $k/n,k=0,1,...,n$. However, in real problems, data at equally
spaced nodes are sometimes contaminated by random errors due to a variety of factors.

Thus, in this section, we consider the approximation in capacity of a random
function $f(x,\omega)$ by stochastic Bernstein polynomials
\[
B_{n}(f,Y)(x,\omega)=\sum_{k=0}^{n}f(Y_{n,k}(\omega),\omega)p_{k,n}(x),
\]
where $Y=\{Y_{n,k}:n\in\mathbb{N},~k=0,...,n\}$ is a triangular array of
random variables $Y_{n,k}:\Omega\rightarrow\mathbb{R}$, such that
\[
0\leq Y_{n,0}\leq Y_{n,1}\leq....\leq Y_{n,n}.
\]

We shall need the following two quantities associated to a random function
$f$:%
\begin{equation}
\omega_{1,x}(f;\delta)(\omega)=\sup\{|f(x,\omega)-f(y,\omega)|:x,y\in
\lbrack0,1],\ |x-y|\leq\delta\}\label{eqomega}%
\end{equation}
and
\begin{equation}
K(f;\delta)=\sup\{\omega_{1,x}(f;\delta)(\omega):\omega\in\Omega\},\quad
\delta\geq0.\label{eqK}%
\end{equation}
It is immediate that $K(f;\delta)$ is nondecreasing and subadditive as
function of $\delta\geq0$. Also, it is easy to see that if $f(x,\omega)$ is
continuous at each $x\in\lbrack0,1],$ uniformly with respect to $\omega$, then
$\lim_{\delta\rightarrow0}K(f;\delta)=0$, so in this case $K(f;\cdot)$ is a
modulus of continuity.

We put
\[
M_{n}(\omega)=\max\left\{  \left\vert Y_{n,k}(\omega)-\frac{k}{n}\right\vert
:0\leq k\leq n\right\}  ,\quad n\in\mathbb{N},\ \omega\in\Omega.
\]

\begin{theorem}
\label{thm5} Let $f:[0,1]\times\Omega\rightarrow\mathbb{R}$ be continuous at
each $x\in\lbrack0,1]$, uniformly with respect to $\omega\in\Omega$,
$\mathcal{C}$ a $\sigma$-algebra of subsets of $\Omega$ and $\mu
:\mathcal{C}\rightarrow\lbrack0,1]$ a capacity. If
\[
\lim_{n\rightarrow\infty}\mu(\{\omega\in\Omega:M_{n}(\omega)>\varepsilon
\})=0,
\]
for every $\varepsilon>0,$ then $B_{n}(f,Y)(x,\omega)$ converges to
$f(x,\omega)$ in capacity, uniformly with respect to $x\in\lbrack0,1]$.

In addition, for every $0<\delta<1$ and $n\geq1/\delta^{2}$, we have
\begin{align}
&  \mu\left(  \{\omega\in\Omega:\sup_{x}\Vert B_{n}(f,Y)(x,\omega
)-f(x,\omega)\Vert>(1+c)K(f,\delta)\}\right) \label{5.3}\\
&  \leq\mu(\{\omega\in\Omega:M_{n}(\omega)>\delta\}),\nonumber
\end{align}
where $c$ is the Sikkema constant.
\end{theorem}

\begin{proof}
First, let $B_{n}(f)(x,\omega)=\sum_{k=0}^{n}f\left(  \frac{k}{n}%
,\omega\right)  p_{k,n}(x)$. If we fix $\omega\in\Omega$, then by repeating
the argument used by Sikkema \cite{Sikkema} for the formula (\ref{Sik}), we
immediately get
\begin{multline}
\sup_{x}\Vert B_{n}(f)(x,\omega)-f(x,\omega)\Vert\ \\
\leq c\cdot\sup\{|f(x,\omega)-f(y,\omega)|:x,y\in\lbrack0,1],~|x-y|\leq
\frac{1}{\sqrt{n}}\}\nonumber\\
\leq c\cdot K\left(  f;\frac{1}{\sqrt{n}}\right)  ,\text{\quad for all }%
\omega\in\Omega,~n\in\mathbb{N}.\nonumber
\end{multline}

For $x\in\lbrack0,1]$ and $\omega_{0}\in\Omega$,\ by the triangular inequality
it follows that
\begin{gather}
|B_{n}(f,Y)(x,\omega_{0})-f(x,\omega_{0})|=\left\vert \sum_{k=0}^{n}%
[f(Y_{n,k}(\omega_{0}),\omega_{0})-f(x,\omega_{0})]p_{k,n}(x)\right\vert
\label{aux1}\\
\leq\sup_{x}\Vert B_{n}(f)(x,\omega_{0})-f(x,\omega_{0})\Vert+\sum_{K=0}%
^{n}\omega_{1,x}(f;|Y_{n,k}(\omega_{0})-k/n|)(\omega_{0})\cdot p_{k,n}%
(x)\nonumber\\
\leq c\cdot K\left(  f;\frac{1}{\sqrt{n}}\right)  +\omega_{1,x}(f;M_{n}%
(\omega_{0}))(\omega_{0})\nonumber\\
\leq c\cdot K\left(  f;\frac{1}{\sqrt{n}}\right)  +K(f;M_{n}(\omega
_{0})),\nonumber
\end{gather}
since
\[
\omega_{1,x}(f;M_{n}(\omega_{0}))(\omega_{0})=\sup\{|f(x,\omega_{0}%
)-f(y,\omega_{0})|:x,y\in\lbrack0,1],~|x-y|\leq M_{n}(\omega_{0})\}
\]%
\[
\leq\sup\{\sup\{|f(x,\omega)-f(y,\omega)|:\omega\in\Omega\}:x.y\in
\lbrack0,1],~|x-y|\leq M_{n}(\omega_{0})\}
\]%
\[
=\sup\{\sup\{|f(x,\omega)-f(y,\omega)|:x,y\in\lbrack0,1],~|x-y|\leq
M_{n}(\omega_{0})\}:\omega\in\Omega\}
\]%
\[
=K(f;M_{n}(\omega_{0})).
\]
For $\varepsilon\in(0,K(f;1))$ and $K\left(  f;\frac{1}{\sqrt{n}}\right)
\leq\varepsilon$, from the previous estimate, by the monotonicity of $\mu$ and
by Lemma 1 in Adell and C\'{a}rdenas-Morales \cite{Adell_2}, we infer that
\[
\mu(\{\omega_{0}\in\Omega:\sup_{x}\Vert B_{n}(f,Y)(x,\omega_{0})-f(x,\omega
_{0})\Vert>(1+c)\varepsilon\})
\]%
\begin{equation}
\leq\mu(\{\omega_{0}\in\Omega;K(f;M_{n}(\omega_{0}))>\varepsilon
\})=\mu(\{\omega_{0}\in\Omega:M_{n}(\omega_{0})>\tilde{K}(f;\varepsilon
)\},\label{aux2}%
\end{equation}
where $\tilde{K}$ is the right-continuous inverse of the modulus of continuity
$K$, given by formula $(\ref{eqK})$ and satisfying%
\begin{equation}
\delta\leq\tilde{K}(f;K(f;\delta)),\quad\text{for all }0\leq\delta
\leq1.\label{aux3}%
\end{equation}
Indeed, the equality in (\ref{aux2}) follows immediately from the
nondecreasing monotonicity of $K(f;\cdot)$ and $\tilde{K}(f;\cdot)$ and
applying $\tilde{K}(f;\cdot)$ to $K(f;M_{n}(\omega_{0}))>\varepsilon$ and
$K(f;\cdot)$ to $M_{n}(\omega_{0})>\tilde{K}(f;\varepsilon)$.

Choosing now
\[
\varepsilon=K(f;\delta)\text{ with }\delta\geq\frac{1}{\sqrt{n}},
\]
in (\ref{aux2}), taking into account (\ref{aux3}) and from $K\left(  f;\frac{1}{\sqrt{n}}\right)  \leq\varepsilon$, it follows the
estimate in the statement.
\end{proof}
\begin{remark}
If $\mu$ is a probability measure and $f$ is deterministic, then $K(f,\delta)$ becomes the usual modulus
of continuity and by Theorem $\ref{thm5}$ we obtain Theorem $1$ in Adell and
C\'{a}rdenas-Morales \cite{Adell_2}.
\end{remark}

In the next lemma we shall need the triangular array $Y$ obtained as follows.
For each $n\in\mathbb{N}$, let $(V_{j})_{j=1}^{n+1}$ be a finite sequence of
independent identically distributed random variables having the uniform
distribution on $[0,1]$. Let%
\[
V_{n+1:1}\leq\cdots\leq V_{n+1:n+1}%
\]
be the order statistics obtained by arranging $(V_{j})_{j=1}^{n+1}$ in
increasing order and put
\begin{equation}
Y=\{Y_{n,k}=V_{n+1:k+1},~n\in\mathbb{N},~k=0,1,...,n\}.\label{eq5.x}%
\end{equation}
\begin{lemma}
\label{lemma5.1} Suppose that $\mu:\mathcal{C}\rightarrow\lbrack0,1]$ is a
distorted probability of the form $\mu=u\circ P$, where $P$ is a probability
measure and $u:[0,1]\rightarrow\mathbb{R}$ is a strictly increasing and
concave function such that $u(0)=0$ and $u(1)=1$. If $0<u^{\prime}(0)<\infty$,
then for every $\varepsilon>0$, $n\in \mathbb{N}$ and $0<r<1$ we have%
\[
\mu(\{\omega\in\Omega:M_{n}(\omega)>\varepsilon\})\leq u^{\prime}(0)\cdot
\frac{n+1}{\sqrt{1-r}}\cdot\exp\left(  -\frac{3r}{2}n\varepsilon^{2}\right)  ,
\]
where $exp$ denotes the exponential function and $M_{n}(\omega)$ is defined
by
\[
M_{n}(\omega)=\max\left\{  \left\vert Y_{n,k}(\omega)-\frac{k}{n}\right\vert
:0\leq k\leq n\right\}  \quad\text{for }n\in\mathbb{N},~\omega\in\Omega.
\]
\
\end{lemma}
\begin{proof} Clearly,
\[
x\leq u(x)\leq u^{\prime}(0)x\text{\quad for all }x\in\lbrack0,1],
\]
which for $x=\mu(A)$ gives us
\[
\mu(A)\leq u^{\prime}(0)P(A)\quad\text{for all }A\in\mathcal{C}.
\]
Combining this fact with the estimate%
\[
P(M_{n}>\varepsilon)\leq\frac{n+1}{\sqrt{1-r}}\exp\left(  -\frac{3r}%
{2}n\varepsilon^{2}\right)  ,
\]
proved by Adell and C\'{a}rdenas-Morales \cite{Adell_2}, Lemma 2, p. 7, we get
the conclusion of Lemma \ref{lemma5.1}.
\end{proof}
\begin{theorem}
\label{thm6} Let $(\tau(n))_{n}$ satisfying the conditions
\begin{equation}
\lim_{n\rightarrow\infty}\tau(n)=\infty,~\lim_{n\rightarrow\infty}\frac
{\tau(n)}{n}=0\text{ and }\tau(n)\geq1\text{ for }n\in\mathbb{N},\label{eq5.4}%
\end{equation}
let $Y$ be the triangular array \emph{(\ref{eq5.x})} and let $\mu$ the
distorted probability defined as in Lemma $\ref{lemma5.1}$. Then, for any
random function $f(x,\omega)$ continuous at each $x\in\lbrack0,1]$ uniformly
with respect to $\omega,$ for any $r\in(0,1)$ and any $n\in\mathbb{N}$, we
have%
\begin{align}
&  \mu\left(  \left\{  \omega\in\Omega:\sup_{x}\Vert B_{n}(f,Y)(x,\omega
)-f(x,\omega)\Vert>(1+c)K\left(  f;\sqrt{\frac{\tau(n)}{n}}\right)  \right\}
\right)  \label{eq5.5}\\
&  \leq u^{\prime}(0)\cdot\frac{n+1}{\sqrt{1-r}}\exp\left(  -\frac{3r}{2}%
\tau(n)\right)  .\nonumber
\end{align}
Here $c$ is the Sikkema's constant.
\end{theorem}

\begin{proof}
Choosing $\delta=\sqrt{\frac{\tau(n)}{n}}$ in Theorem \ref{thm5}, we have
$n\geq\frac{1}{\delta^{2}}$ since $\tau(n)\geq1.$ The proof ends by applying
Lemma \ref{lemma5.1}.
\end{proof}
There are many examples of distorted probabilities $\mu=u\circ P$ satisfying
the hypothesis of Theorem \emph{\ref{thm6}}. One can choose $u(t)=\frac
{2t}{t+1},~u(t)=(1-e^{-t})/(1-e^{-1}),$ $u(t)=\ln(1+t)/\ln(2)$,
$u(t)=\sin(\pi t/2)$, or $u(t)=\frac{4}{\pi}\cdot \arctan(t)$, for $t\in\lbrack0,1]$.

If $\mu$ is a probability measure (that is, when $u(t)=t$) and $f$ is a
deterministic function, then Theorem \emph{\ref{thm6}} reduces to Corollary
$1$ in Adell and C\'{a}rdenas-Morales \cite{Adell_2}.

\section*{Conflict of interest}

The authors declare that they have no conflict of interest.


\begin{thebibliography}{99}                                                                                               %


\bibitem {Adell}Adell, J.A., Bustamante, J., Quesada, J.M.: Estimates for the
moments of Bernstein polynomials. J. Math. Anal. Appl. \textbf{432}, 114-128 (2015)

\bibitem {Adell_2}Adell, J.A., C\'{a}rdenas-Morales, D.: Stochastic Bernstein
polynomials: uniform convergence in probability with rates. Adv. Comput. Math.
\textbf{46}, Art. 16, 10 pages (2020)

\bibitem {Cen-2}Cenu\c{s}\u{a}, Gh., S\u{a}cuiu, I.: On some stochastic
approximations for random functions. Rend. Mat., Serie VI \textbf{12}, 143-156 (1979)

\bibitem {Cerda}Cerd\`{a}, J., Mart\'{\i}n, J., Silvestre P.: Capacitary
function spaces. Collect. Math. \textbf{62}, 95-118 (2011)

\bibitem {Cooman}De Cooman, G., Kerre, E.E., Vanmassenhove, F.: Possibility
theory: An integral theoretic approach. Fuzzy Sets and Systems \textbf{46},
287-300 (1992)

\bibitem {Denn}Denneberg, D.: Non-Additive Measure and Integral. Kluwer
Academic Publisher, Dordrecht (1994)

\bibitem {Dubois}Dubois, D., Prade, H.: Th\'{e}orie des Possibilit\'{e}s.
Masson, Paris (1985)

\bibitem {FS}F\"{o}llmer, H., Schied, A.: Stochastic Finance, Fourth revised
and extended edition. De Gruyter (2016)

\bibitem {Gal-Random-1}Gal, S.G.: Jackson type estimates in the approximation
of random functions by random polynomials. Rend. Mat. Appl. (7) \textbf{14}%
(4), 543-556 (1994)

\bibitem {Gal-Random-2}Gal, S.G.: Approximation theory in random setting.
Chapter 12 in Handbook of analytic-computational methods in applied
mathematics, pp. 571-616. Chapman and Hall/CRC, Boca Raton, FL (2000).

\bibitem {Gal2008}Gal, S.G.: Shape-Preserving Approximation by Real and
Complex Polynomials. Birkh\"{a}user, Boston (2008)

\bibitem {Gal5}Gal, S.G.: Shape preserving properties and monotonicity
properties of the sequences of Choquet type integral operators. J. Numer.
Anal. Approx. Theory \textbf{47}(2), 135-149 (2018)

\bibitem {Gal-Nic-2}Gal, S.G., Niculescu, C.P.: A nonlinear version of
Korovkin's theorem. Mediterr. J. Math. \textbf{17}(5), article no. 145 (2020)

\bibitem {Gal-Random-3}Gal, S.G., Villena, A.R.: Random condensation of
singularities and applications. Random Oper. Stochastic Equations
\textbf{5}(3), 263-268 (1997)

\bibitem {Gr}Grabisch, M.: Set Functions, Games and Capacities in Decision
Making. Theory and Decision Library C (Game Theory, Social Choice, Decision
Theory, and Optimization), vol \textbf{46}, Springer (2016).

\bibitem {Ignatov}Ignatov, Z.G., Mills, T.M., Tzankova, I.P.: On the rate of
approximation of random functions. Serdica, Bulgaricae mathematicae
publicationes \textbf{18}, 240-247 (1992)

\bibitem {Kamolov}Kamolov, A.I.: On exact estimates of approximation of random
processes (in Russian). Dokl. Akad. Nauk. UzSSR \textbf{11}, 4-6 (1986)

\bibitem {Lor}Lorentz, G.G.: Bernstein Polynomials, Second edition. Chelsea
Publishing Company, New York (1986)

\bibitem {Onic-2}Onicescu, O., Istr\u{a}\c{t}escu, V.I.: Approximation
theorems for random functions. Rend. Mat., Serie VI \textbf{8}(1), 65-81 (1975)

\bibitem {Onic-3}Onicescu, O., Istr\u{a}\c{t}escu, V.I.: Approximation
theorems for random functions, II. Rend. Mat. Serie VI \textbf{11}(4), 585-589 (1978)

\bibitem {Sikkema}Sikkema, P.C.: Der Wert einiger Konstanten in der Theorie
der Approximation mit Bernstein- Polynomen. Numer. Math. \textbf{3}, 107-116 (1961)

\bibitem {Sun}Sun, X., Wu, Z.: Chebyshev type inequality for stochastic
Bernstein polynomials. Proc. Amer. Math. Soc. \textbf{147}(2), 671-679 (2019).

\bibitem {Zadeh}Zadeh, L.A.: Fuzzy sets as a basis for a theory of
possibility. Fuzzy Sets and Systems \textbf{1}, 3-28 (1978)

\bibitem {Klir}Wang, Z., Klir, G.J.: Generalized Measure Theory. Springer, New
York (2009)

\bibitem {Wu_1}Wu, Z., Sun, X., Ma, L.: Sampling scattered data with Bernstein
polynomials: stochastic and deterministic error estimates. Adv. Comput. Math.
\textbf{38}, 187-205 (2013)

\bibitem {Wu_2}Wu, Z., Zhou, X.: Polynomial convergence order of stochastic
Bernstein approximation. Adv. Comput. Math. \textbf{46}, Art. 8, 14 pages (2020)
\end{thebibliography}
\end{document}